\documentclass[11pt]{amsart}
\usepackage{graphicx,amssymb}

\usepackage{longtable}

\usepackage{tikz}

\usepackage[square,sort&compress,comma,numbers]{natbib}
\usepackage{nicefrac,xcolor,upref,version}

\usepackage[colorlinks=true,citecolor=cyan,backref=page]{hyperref}

\usepackage{amscd,amsthm,amsmath,amssymb,amsfonts}

\usepackage{mathrsfs}

\newcommand{\M}{\mathfrak{M}}

\newtheorem{theorem}{Theorem}[section]
\newtheorem{lemma}[theorem]{Lemma}

\newtheorem{problem}[theorem]{Problem}

\numberwithin{equation}{section}

\def\Q{{\mathbb {Q}}}
\def\Z{{\mathbb Z}}  
\def\R{{\mathbb R}} 
\def\C{{\mathbb C}} 
\def\eps{{\varepsilon}}

\def\undx{{\underline x}}
\def\undk{{\underline k}}
\def\undt{{\underline t}}
\def\undb{{\underline b}}
\def\undy{{\underline y}}

\def\cL{{\mathcal L}}

\def\beq{\begin{equation}}
\def\eeq{\end{equation}}

\def\Norm{{\rm Norm}} 
\def\Gal{{\rm Gal}}

\def\bfx{{\bf x}} 
\def\tih{{\widetilde h}} 
\def\rmi{{\rm i}}
\def\resp{{\it resp., }}

\def\otau{{\overline{\tau}}}

\begin{document}

\title[Effective approximation to complex algebraic numbers]{Effective approximation to complex algebraic numbers
by algebraic numbers of bounded degree}

\author{Prajeet Bajpai}
\address{Department of Mathematics, University of British Columbia, Vancouver, B.C., V6T 1Z2 Canada}
\email{prajeet@math.ubc.ca}

\author{Yann Bugeaud}
\address{I.R.M.A., UMR 7501, Universit\'e de Strasbourg
et CNRS, 7 rue Ren\'e Descartes, 67084 Strasbourg Cedex, France}
\address{Institut universitaire de France}
\email{bugeaud@math.unistra.fr}

\begin{abstract}
We establish the first effective improvements on the Liouville inequality 
for approximation to complex non-real algebraic numbers by complex 
algebraic numbers of degree at most $4$. 
\end{abstract}


\maketitle

\section{Introduction} \label{intro}

Throughout this paper, the (na\"\i ve) height $H(\alpha )$ of an algebraic number $\alpha$
is the (na\"\i ve) height $H(P)$ of its minimal defining polynomial $P(X)$ over $\Z$, that is, the 
maximum of the absolute values of the coefficients of $P(X)$. 
Let $\xi$ be an algebraic real number of degree $d \ge 2$. 
Let $n$ be a positive integer. 
By a Liouville-type argument (see e.g. \cite{Gu67}), there exists a positive number $c_1(\xi, n)$ such that
\beq  \label{liouv}
|\xi - \alpha| > c_1 (\xi, n) H(\alpha)^{-d}, \quad \hbox{for every $\alpha \not= \xi$ algebraic of degree $\le n$}. 
\eeq
In the opposite direction, Wirsing \cite{Wir61} established in 1961 that there exists a positive number $c_2(\xi)$ such that
$$
|\xi - \alpha| < c_2 (\xi) H(\alpha)^{-d}, \quad 
\hbox{for i. m. $\alpha$ algebraic of degree $\le d-1$,}
$$
where i. m. stands for `infinitely many', 
see also \cite[Section 2.4]{BuLiv}. Thus, approximation to $\xi$ by algebraic numbers of degree at most $n$ is well 
understood when $n \ge d-1$. The case $n \le d-2$ is much more difficult. 

For $n=1$, \eqref{liouv} was considerably improved by Roth \cite{Ro55}, who showed that, for any $\eps > 0$,
there exists a positive number $c_3(\xi, \eps)$ such that
\beq  \label{roth}
|\xi - p/q| > c_3 (\xi, \eps) H(p/q)^{-2 - \eps}, \quad \hbox{for every rational number $p/q$}. 
\eeq
More generally, Schmidt \cite{Schm70,SchmLN} established that, for any $\eps > 0$ and any positive integer $n$ with $n \le d-2$,
there exists a positive number $c_4(\xi, \eps, n)$ such that
\beq  \label{schm}
|\xi - \alpha| > c_4 (\xi, \eps, n) H(\alpha)^{-n -1 - \eps}, \quad \hbox{for every $\alpha$ algebraic of degree $\le n$}. 
\eeq 
By a suitable transference argument, this implies that for any $\eps > 0$ and any positive integer $n$ with $n \le d-2$,
there exists a positive number $c_5(\xi, \eps, n)$ 
such that
\beq  \label{schmtrans}
|\xi - \alpha| < c_5 (\xi, \eps, n) H(\alpha)^{-n -1 + \eps}, \quad 
\hbox{for i. m. $\alpha$ algebraic of degree $\le n$.} 
\eeq 
A slight improvement has been established in \cite{Bu09}, where $\eps$ in \eqref{schmtrans} 
is replaced by a function of $H(\alpha)$ which tends to $0$ 
as $H(\alpha)$ tends to infinity. 

Unlike $c_1 (\xi, n)$, the positive numbers $c_3(\xi, \eps)$ and $c_4(\xi, \eps, n)$ are not effectively computable. 
Thus, there is a big gap between the effective result \eqref{liouv} and the ineffective results \eqref{roth} and \eqref{schm}. 
For $n=1$, that is, for rational approximation to algebraic numbers of degree $\ge 3$, this gap 
has been slightly reduced by Feldman \cite{Fe71} (see also \cite{Ba73}), 
as an application of the theory of linear forms in logarithms, first developed by Alan Baker \cite{Ba66}. 

\begin{theorem}[Feldman]  \label{feld}
For any real algebraic number $\xi$ of degree $d \ge 3$, 
there exist effectively computable positive numbers $\kappa(\xi)$ and $c_6(\xi)$ such that
$$
|\xi - p/q| > c_6 (\xi) H(p/q)^{ - d + \kappa (\xi) }, \quad \hbox{for every rational number $p/q$}. 
$$
\end{theorem} 

Bombieri \cite{Bo93} has given an alternative proof of Theorem \ref{feld}, independent of Baker's theory. 
Both methods yield a very small value for $\kappa (\xi)$. 
For some specific $\xi$, it is possible to improve Theorem \ref{feld}; see \cite[Section 4.10]{Bu18b} 
for references.  We do not know any examples of real algebraic numbers $\xi$ of degree $d \ge n+2$ 
for which \eqref{liouv} can be improved effectively, for any given $n\ge 2$.

Approximation to complex non-real algebraic numbers $\xi$ of degree $d \ge 2$ by 
complex algebraic numbers of degree at most $n$ was studied in 2009 by Bugeaud and Evertse \cite{BuEv09}. 
As in the real case, there exist positive numbers $c_7(\xi, n)$ and $c_8 (\xi)$ such that
\beq  \label{liouvC}
|\xi - \alpha| > c_7 (\xi, n) H(\alpha)^{- \frac{d}{2}},  \quad  \hbox{for every $\alpha \not= \xi$ algebraic of degree $\le n$},
\eeq
and
$$
|\xi - \alpha| < c_8 (\xi) H(\alpha)^{-\frac{d}{2}}, \quad 
\hbox{for i. m. $\alpha$ algebraic of degree $\le d-1$,}
$$
see e.g. \cite[Proposition 10.2]{BuEv09}. The more difficult case $n \le d-2$ is not fully understood. 
It has been shown in \cite{BuEv09} that there exists $w_n^* (\xi)$ in $\{(n-1)/2, n/2\}$ 
and, for every $\eps > 0$, there are positive numbers $c_9 (\xi, \eps, n)$ and $c_{10} (\xi, \eps, n)$
such that 
\begin{align*}
|&\xi - \alpha| > c_9 (\xi, \eps, n) H(\alpha)^{-w_n^* (\xi) - 1 - \eps}, \quad  \hbox{for every $\alpha$ algebraic of degree $\le n$}, \\
|&\xi - \alpha| < c_{10} (\xi, \eps, n) H(\alpha)^{-w_n^* (\xi) - 1 + \eps}, \quad 
\hbox{for i. m. $\alpha$ algebraic of degree $\le n$.}
\end{align*}
The exact value of $w_n^* (\xi)$ has been determined in \cite{BuEv09}, 
except when $n \ge 6$ is even, $n+2 < d \le 2n - 2$, $[\Q(\xi) : \Q(\xi) \cap \R] = 2$, and $1, \xi + \overline{\xi}, \xi \cdot \overline{\xi}$ are linearly independent over $\Q$.
Here and below, ${\overline {\ \cdot \ }}$ denotes complex conjugation.  In particular, we know that $w_n^* (\xi) = (n-1)/2$ for every complex non-real algebraic number $\xi$ and 
every odd integer $n$. Except \eqref{liouvC} all these results on approximation to complex non-real algebraic numbers are ineffective 
and, to the best of our knowledge,  there has been up to now no contribution to the 
following problem.

\begin{problem}   \label{pbC} 
To find examples of integers $n \ge 2$ and complex
non-real algebraic numbers $\xi$ of degree $d \ge n+2$ for which an effective improvement of \eqref{liouvC} holds.
\end{problem}

The main purpose of the present paper is to answer Problem \ref{pbC} for any integer $n$ in $\{2, 3, 4\}$, thereby 
giving the first examples of complex, non-real algebraic numbers $\xi$ of degree $d \ge n+2$ for which \eqref{liouvC} 
can be effectively improved.

\section{Results}

A totally complex (or totally imaginary) algebraic number is a complex algebraic number, none of whose 
Galois conjugates is real. 
A number field $K$ is a CM-field if it is a quadratic extension $K/F$ such that the image of every complex embedding of $F$ 
is contained in $\R$, but there is no complex embedding of $K$ whose image is contained in $\R$.   
We now state our main results. 

\begin{theorem}  \label{thw2} 
Let $\xi$ be a totally complex algebraic number of degree $d \ge 4$. 
If $d=4$, assume furthermore 
that $\Q(\xi)$ is not a CM-field. Then, there exist effectively computable positive real numbers $\kappa (\xi)$ and $c(\xi)$ such that 
$$
|\xi - \alpha| > c(\xi) H(\alpha)^{ - \frac{d}{2} + \kappa(\xi) }, \quad 
\hbox{for every $\alpha$ algebraic of degree $\le 2$.} 
$$
\end{theorem} 

Theorem \ref{thw2} improves the Liouville inequality \eqref{liouvC} 
for \emph{all} totally complex algebraic numbers for which it is possible to improve \eqref{liouvC}; see the 
beginning of subsection \ref{wellapp}.    

\begin{theorem}  \label{thw3} 
Let $\xi$ be a totally complex algebraic number of degree $d \ge 6$ such that the algebraic number field $\Q (\xi)$ is Galois. 
Then, there exist effectively computable positive real numbers $\kappa (\xi)$ and $c(\xi)$ such that 
$$
|\xi - \alpha| > c(\xi) H(\alpha)^{ - \frac{d}{2} + \kappa(\xi)}, \quad 
\hbox{for every $\alpha$ algebraic of degree $\le 3$.} 
$$
\end{theorem}

Theorem \ref{thw3} improves the Liouville inequality \eqref{liouvC} 
for all totally complex algebraic numbers generating a Galois extension.

\begin{theorem}    \label{thw4} 
Let $\xi$ be a totally complex algebraic number of degree $d \ge 8$  such that the algebraic number field $\Q (\xi)$ is Galois. 
Assume 
that $1, \sigma(\xi) + \overline{\sigma(\xi)}, \sigma(\xi) \cdot \overline{\sigma(\xi)}$ are linearly independent over $\Q$, for every 
Galois conjugate $\sigma(\xi)$ of $\xi$. Then, there exist effectively computable positive real numbers $\kappa (\xi)$ and $c(\xi)$ such that 
$$
|\xi - \alpha| > c(\xi) H(\alpha)^{ - \frac{d}{2} + \kappa(\xi)}, \quad 
\hbox{for every $\alpha$ algebraic of degree $\le 4$.} 
$$
\end{theorem}

We assume $d \ge 8$ in Theorem \ref{thw4} since, by the results of \cite{BuEv09}, the Liouville inequality \eqref{liouvC} 
with $n=4$    
is best possible for every totally complex 
algebraic number $\xi$ of degree $6$ such that the field $\Q (\xi)$ is Galois. 
We believe that the linear independence assumption in Theorem \ref{thw4} is satisfied 
for a generic totally complex algebraic number which generates a Galois extension; see subsection \ref{testing}.

For the proof of Theorem \ref{thw2} (\resp Theorem \ref{thw3}, Theorem \ref{thw4}), we study norm-form equations 
and establish (except when $n=4$ and $d=8$, see below) that there are effectively computable positive numbers $H_0$ and $\delta$ such that
$$
|\Norm_{K / \Q} (P(\xi))| \ge H(P)^\delta,
$$
for every integer polynomial $P(X)$ of degree at most $2$ (\resp at most $3$, at most $4$) of height greater that $H_0$. 
Then, we apply the (easy) Lemma \ref{transf}. By this method we obtain an effective improvement for 
{\it every} Galois conjugate of $\xi$. Thus it is perhaps unsurprising that our theorems apply only to totally 
complex algebraic numbers. 
Furthermore, this explains why the assumption in Theorem \ref{thw4} is about every Galois conjugate of $\xi$.

In Theorems \ref{thw2} and \ref{thw3}, the positive real number $\kappa$ takes the form 
$$
\bigl( (c_1 d)^{c_2 d} R_K  \log \max\{3, R_K\}  \bigr)^{-1},
$$ 
where $R_K$ is the regulator of the number field $\Q(\xi)$ and $c_1, c_2$ are (large) real numbers, independent of $\xi$. 
In Theorem \ref{thw4}, it takes the form 
$$
\bigl( (c_1 d)^{c_2 d} R_K^2  \log \max\{3, R_K\}   \bigr)^{-1},      
$$
since we have two applications of Theorem \ref{lflog}. 

The proofs of Theorems \ref{thw2} to \ref{thw4} are based on the recent results of Bajpai \cite{Baj23} 
on the effective resolution of norm-form equations in three to five variables, which considerably 
extend earlier results of Vojta \cite{Vo83}. Except when $n=4$ and $d=8$, we prove a stronger version of these results, suitable for our intended application. 
In addition, our treatment of norm-form equations in five variables is completely new; 
indeed we prove the following result which may be  
of independent interest.

\begin{theorem}    \label{thnf4} 
Let $\xi$ be a totally complex algebraic number of degree $d \ge 10$  such that the algebraic number field $K = \Q (\xi)$ is Galois. 
Assume 
that $1, \sigma(\xi) + \overline{\sigma(\xi)}, \sigma(\xi) \cdot \overline{\sigma(\xi)}$ are linearly independent over $\Q$, for every 
Galois conjugate $\sigma(\xi)$ of $\xi$. Then, there exist effectively computable positive real numbers $\kappa (\xi)$ and $c(\xi)$ such that 
$$
| \Norm_{K/\Q}(x_0 + x_1\xi + x_2\xi^2 + x_3\xi^3 + x_4\xi^4 ) | \; > \; c(\xi) X^{ \kappa(\xi)},
$$
where $X = \max\{ |x_0|, \ldots, |x_4| \}$.
\end{theorem}

In particular, for $\xi$ satisfying the conditions of the theorem, we can effectively solve the norm-form equation
\[
\Norm_{K/\Q}(x_0 + x_1\xi + x_2\xi^2 + x_3\xi^3 + x_4\xi^4 ) \; = \; m
\]
for any fixed integer $m$.

Quite surprisingly, Theorem \ref{thnf4} is no longer true for $d=8$. 
We give below an explicit counter-example. 
Let $\xi$ be a root of the polynomial $x^8 + 7x^4 + 1$. 
Then, the field $\Q(\xi)$ is Galois, CM, and totally complex. Furthermore, 
we check that $\xi$ satisfies the linear independence criteria of Theorem \ref{thnf4}. 
Now, $\xi^4$ is quadratic with minimal polynomial $x^2 + 7x + 1$ and it is easy to see that any integral power of $\xi^4$ 
is of the form $x_0 + x_4  \xi^4$, where $x_0$ and $x_4$ are integers. Consequently, 
we get infinitely many (quadratic) solutions to the norm form equation
$$
| \Norm_{\Q(\xi) /\Q}(x_0 + x_1\xi + x_2\xi^2 + x_3\xi^3 + x_4\xi^4 ) | = 1. 
$$
However,  these solutions to the norm-form equation do not give rise to very good approximations to $\xi$, hence there is no contradiction 
with Theorem \ref{thw4}.

The rest of the paper is organized as follows. We comment on our results in the next section 
and justify the claims made below our first two theorems. 
Some auxiliary results are gathered in Section \ref{auxres}, while Section \ref{proofs} contains the proofs of 
Theorems \ref{thw2} to \ref{thnf4}.

\section{Discussion}

\subsection{Unusually well-approximated numbers}   \label{wellapp}   
Let $r, s$ be positive integers and assume that $r$ is not a perfect square. 
Let $a, b$ be positive integers such that 
\beq \label{pell}
a^2 - r b^2 = 1, 
\eeq
and set
$$
P_{a, b} (X) = b X^2 - 2 a X + (r + s) b.
$$
Since
$$
| P_{a,b} (\sqrt{r} + \rmi \sqrt{s} ) | \le \frac{2 \sqrt{r+s}}{a + b \sqrt{r}} \le \frac{2s (r+s)}{H(P_{a,b})},
$$
the polynomial $P_{a, b} (X)$ has a root $\alpha_{a, b}$ such that 
$$
| (\sqrt{r} + \rmi \sqrt{s} ) - \alpha_{a,b} | \le \frac{2 (r+s)^2}{H(\alpha_{a,b})^2}.
$$
Since \eqref{pell} has infinitely many integer solutions, we conclude that the Liouville inequality
\eqref{liouvC} is best possible (up to the numerical constant) for the quartic  
number $\sqrt{r} + \rmi \sqrt{s}$ and $n=2$. More generally, 
\cite[Corollary 2.4]{BuEv09} states that, when $n=2$,  \eqref{liouvC} 
is best possible for the complex non-real algebraic number $\xi$ of degree $\ge 4$ if and only if $\xi$ has degree $4$ and 
$1, \xi + \overline{\xi}, \xi \cdot \overline{\xi}$ are linearly dependent over $\Q$. 
As pointed out to us by the referee, the latter condition is equivalent to 
saying that $\Q(\xi)$ is a quartic CM-field;  see Lemma \ref{CM} below.   
Thus Theorem \ref{thw2} improves the Liouville inequality 
for all totally complex algebraic numbers where it is possible to improve \eqref{liouvC} with $n=2$.

For the case $n=4$, again there are situations where \eqref{liouvC} is best possible.
Indeed, by \cite[Corollary 2.4 and Proposition 10.3]{BuEv09}, when $n=4$, the Liouville inequality \eqref{liouvC} 
is best possible (up to the numerical constant) for the complex non-real algebraic number $\xi$ of degree $d \ge 6$ 
if and only if $\xi$ has degree $6$ and $[\Q(\xi) : \Q(\xi) \cap \R]  = 2$ or $\xi$ has degree $6$ and 
$1, \xi + \overline{\xi}, \xi \cdot \overline{\xi}$ are linearly dependent over $\Q$. 
If the totally complex algebraic number field $\Q(\xi)$ is Galois and $d=6$, 
then $\Q(\xi)$ has a subfield of degree $3$ fixed by complex conjugation, hence real. We then have 
$[\Q(\xi) : \Q(\xi)\cap \R] = 2$ and know that \eqref{liouvC} cannot be improved in this case. 

As a practical example, consider the polynomial
$$
f(X) = X^6 - 3 X^5 + 5 X^4 - 5 X^3 + 5 X^2 - 3 X + 1.
$$
If $\xi$ is a root of $f$, then $K = \Q(\xi)$ is a totally complex Galois sextic field (LMFDB label 6.0.12167.1).  
We can label the roots of $f$ as 
$\xi_1, \overline{\xi_1}, \xi_2, \overline{\xi_2}, \xi_3, \overline{\xi_3}$ in such a way that 
$\xi_1 \cdot \overline{\xi_1} = 1$, $\xi_2 + \overline{\xi_2} = 1$, and 
$\xi_3 + \overline{\xi_3} = \xi_3 \cdot \overline{\xi_3}$. 
This shows that $1, \xi + \overline{\xi}, \xi \cdot \overline{\xi}$ are linearly dependent over $\Q$, for any root $\xi$ of $f$, 
thus the Liouville inequality is best possible for all roots of $f$. 
We note it is possible to completely solve the norm-form equation
$$
\Norm_{K / \Q} (x_0 + x_1 \xi + \ldots + x_4 \xi^4) = \pm 1,
$$
using the methods of \cite{Baj23}. Indeed, the equation has $3$ infinite families of solutions and $120$ exceptional solutions. 
The computational details of this and other similar examples will be discussed more thoroughly in \cite{Bajcomp23}.

\subsection{Comparing $\xi$ with its Galois conjugates}   \label{galconj}
By the method used to prove Theorems \ref{thw2} to \ref{thw4}, if we manage to improve 
\eqref{liouvC} in an effective way for a totally complex algebraic number $\xi$ and some $n \ge 2$, then we have a similar 
improvement for each of the Galois conjugates of $\xi$. This is clear from Lemma \ref{transf}. 
However, the results of \cite{BuEv09} show that the value of $w_n^* (\xi)$, which is equal either to $(n-1)/2$ or to $n/2$, depends 
strongly on whether
$$
1, \xi + \overline{\xi}, \xi \cdot \overline{\xi}
$$
are linearly dependent over $\Q$ or not. This motivates the following question. 

\begin{problem}     \label{wnconj} 
Let $n \ge 2$ be an even integer. 
Do there exist complex non-real algebraic numbers $\xi$ having a Galois conjugate $\sigma(\xi)$ such that 
$$
w_n^* (\xi) = \frac{n-1}{2} \quad \hbox{and} \quad w_n^*(\sigma(\xi)) = \frac{n}{2}  \, ? 
$$
\end{problem}

The answer is positive for $n=2$ and $n=4$. Indeed, consider the polynomial 
$$
g(X) = X^6 - X^5  + 2 X^3 -  X + 1, 
$$
whose LMFDB \cite{LMFDB} label is 6.0.27556.1. Its Galois closure has degree $48$. 
We label the roots of $g$ as 
$\xi_1, \overline{\xi_1}, \xi_2, \overline{\xi_2}, \xi_3, \overline{\xi_3}$ in such a way that 
$\xi_1 \cdot \overline{\xi_1} = 1$. 
We check that $1, \xi_2 + \overline{\xi_2}, \xi_2 \cdot \overline{\xi_2}$ and 
$1, \xi_3 + \overline{\xi_3}, \xi_3 \cdot \overline{\xi_3}$ are linearly independent over $\Q$. 
Furthermore, we check that the number fields $\Q(\xi_2)$ and $\Q(\xi_3)$ do not have a non-trivial real subfield.
Consequently, it follows from \cite[Corollary 2.4]{BuEv09} that 
$$
w_2^* (\xi_1) = 1 \quad \hbox{and} \quad w_2^* (\xi_2) = w_2^* (\xi_3) = \frac{1}{2}, 
$$
and
$$
w_4^* (\xi_1) = 2 \quad \hbox{and} \quad w_4^* (\xi_2) = w_4^* (\xi_3) = \frac{3}{2}.
$$
This means that the method developed in the present paper 
cannot be used to get an effective improvement of \eqref{liouvC} 
for $\xi_2$, nor for $\xi_3$.  While we do not investigate Problem \ref{wnconj} further, we believe that it has a positive answer 
for every even $n \ge 6$.

\subsection{Testing the linear independence criterion}  \label{testing}
In Theorems \ref{thw4} and \ref{thnf4} we consider algebraic numbers $\xi$ such that, 
for all Galois conjugates $\sigma(\xi)$ of $\xi$, the numbers
\[
1,  \sigma(\xi)+\overline{\sigma(\xi)},  \sigma(\xi)\cdot\overline{\sigma(\xi)}
\]
are $\Q$-linearly independent. For a given algebraic $\xi$, this criterion can be checked directly 
by computing the action of the Galois group of the minimal polynomial of $\xi$. 
Clearly linear independence does not always hold---for example, it fails if $\xi$ has a Galois conjugate on the unit circle. 
However, one can ask whether this linear independence is expected for a `random' algebraic number; 
we present some numerical evidence in support of this expectation.

For small values of $d$, one can check the LMFDB \cite{LMFDB} for a list of totally complex Galois number fields 
of degree $d$. 
Each field is presented in the form $K = \Q[X]/(f)$. Letting $\gamma$ be a root of this listed $f$, 
one can generate elements in $K$ by taking `random' linear combinations of the powers $1, \gamma, \ldots , \gamma^{d-1}$ of $\gamma$. More specifically, for each $i=0, 1, \ldots , d-1$, one can randomly sample $c_i$ from some fixed interval 
$[-R,R]$ and consider the elements 
\[
\xi = c_0 + c_1\gamma + c_2\gamma^2 + \cdots + c_{d-1}\gamma^{d-1} \; \in \; K.
\]
One can then check how often such an element satisfies the desired linear independence criterion for all its Galois conjugates.

We carried out this process for $d$ in $\{8,10,12,14,16\}$, testing the first 100 fields in each degree 
(ordered by absolute discriminant) and 100 `random' elements as above 
(with coefficients $c_i$ in $[-10,10]$) for each field. Of the 50,000 algebraic numbers thus sampled, 
only one failed to satisfy the conditions of Theorems \ref{thw4} and \ref{thnf4}. 
In particular, the conclusions of Theorems \ref{thw4} and \ref{thnf4} hold for $49,999$ 
of the `randomly generated' algebraic numbers. Finally, we note that it is also possible 
to describe infinite families of $\xi$ of arbitrarily large degree satisfying these conditions. For example, if $d_1,\ldots,d_k$ are co-prime positive squarefree integers and $k\ge 3$, then
\[
\xi \; = \; \sqrt{d_1} + \sqrt{d_2} + \cdots + \sqrt{d_{k-1}} + \sqrt{-d_k}
\]
satisfies the necessary linear independence criterion, and $\Q(\xi)$ is totally complex and Galois.

\subsection{Exponents of Diophantine approximation}  
At the end of Section \ref{intro}, we have defined $w_n^* (\xi)$ 
in accordance with the exponents of Diophantine approximation $w_n^*$ introduced in 1939 
by Koksma \cite{Ko39}, which complement the exponents $w_n$ defined in 1932 by Mahler \cite{Mah32}. 
Let us recall their definitions. Let $n$ be a positive integer and $\theta$ a complex number. 
We let $w_n (\theta)$ be the supremum of the real numbers $w$ for which the inequality 
$$
0 < |P(\theta)| \le H(P)^{-w}
$$
is satisfied by infinitely many integer polynomials $P(X)$ of degree at most $n$. 
We let $w_n^* (\theta)$ be the supremum of the real numbers $w^*$ for which the inequality 
$$
0 < |\theta - \alpha| \le H(\alpha)^{-w^* - 1}
$$
is satisfied by infinitely many algebraic numbers $\alpha$ of degree at most $n$. The reader is directed 
to \cite{BuLiv} for an overview of the known results on $w_n$ and $w_n^*$, including a proof that 
$w_n^* (\theta) \le w_n (\theta)$ always holds.

Let $\xi$ be a complex non-real algebraic number of degree $d$. 
A Liouville-type argument shows that
$$
w_n^* (\xi) \le w_n (\xi) \le \frac{d}{2} - 1.     
$$
Under the assumptions of Theorems \ref{thw2} to \ref{thw4}, we have obtained an effective improvement of the form 
$w_n^* (\xi) \le \frac{d}{2}-1 - \kappa$, for some positive $\kappa$. It should be pointed out that our proof actually gives 
$w_n (\xi) \le \frac{d}{2}-1 - \kappa$, which is a slightly more precise upper bound.

\section{Auxiliary results}   \label{auxres}

Throughout this paper, we let $h$ denote the logarithmic Weil height and we set 
$h_* ( \cdot ) = \max\{h (\cdot ), 1\}$ and $\log_* (\cdot ) = \max\{ \log (\cdot ), 1 \}$. 
It is well-known that the (na\"\i ve) height and the Weil height are comparable, namely, for 
any complex algebraic number $\alpha$ of degree $d$, we have
$$
- \frac{1}{2} \log (d+1) + d h(\alpha) \le   \log H(\alpha) \le d ( h(\alpha) + \log 2),
$$
see e.g. \cite[Chapter 3]{WaLiv}. 
We recall first a classical estimate for linear forms in complex logarithms of algebraic numbers.

\begin{theorem} \label{lflog} 
Let $n \ge 1$ be an integer. 
Let $\alpha_1, \ldots, \alpha_n, \alpha_{n+1}$ be non-zero algebraic numbers.  Let $D$ denote the degree of the algebraic number field generated by $\alpha_1, \ldots, \alpha_{n+1}$ over $\Q$. Let $b_1, \ldots, b_n$ be non-zero integers and set
$$
B = \max\{|b_1|, \ldots, |b_n|\}. 
$$
If $\alpha_1^{b_1} \ldots \alpha_n^{b_n} \alpha_{n+1}  \not= 1$, then we have
$$
\log |\alpha_1^{b_1} \ldots \alpha_n^{b_n} \alpha_{n+1} - 1|  
> - c(n, D) \, h_* (\alpha_1) \cdots h_*(\alpha_{n+1}) \, \log_* \Bigl( \frac{B}{h_* (\alpha_{n+1})} \Bigr). 
$$
\end{theorem}

\begin{proof}
See \cite{Bu18b} or \cite{WaLiv}. 
\end{proof}

A crucial point in Theorem \ref{lflog} is the presence of $h_*(\alpha_{n+1})$ 
in the denominator in the logarithm. This (apparently small) improvement 
is the key for the proof of Theorem \ref{feld} (otherwise, the positive $\kappa (\xi)$
has to be replaced by a function of $H(p/q)$ tending to zero as $H(p/q)$ tends to infinity) and 
of our results; see the survey \cite{Bu23} for more explanations.

We will use the following consequence of Theorem \ref{lflog}. 

\begin{theorem} \label{uniteq}  
Let $\alpha$ be a nonzero element of an algebraic number field $K$. 
Let $v$ be a unit in $K$ with $\alpha v \not= 1$. 
Then, there is an effectively computable constant $C_1 = C_1 (K)$ such that 
$$
\log |1 - \alpha v| \ge - C_1 h_* (\alpha) \log_*  \frac{h_* (v)}{h_* (\alpha)}.
$$
\end{theorem}

\begin{proof}
We use Dirichlet's unit theorem to express $v$ as a product of a root of unity in $K$ times 
integral powers of elements of a fundamental system of units in $K$. 
The theorem then follows from Theorem \ref{lflog}. 
\end{proof}

\begin{theorem} \label{threeuniteq} 
Let $K$ be an algebraic number field of degree $d$ and discriminant $D_K$.
Let $\alpha$, $\beta $ be non-zero elements of $K$.
All solutions $x, y$ in the group of units of $K$ 
to the unit equation
$$
\alpha x + \beta y = 1    
$$
satisfy
\beq \label{bounduniteq}
{\rm max}\{h(x), h(y)\} \ll_{d, D_K}  {\rm max} \{h_* (\alpha ), h_* (\beta )\}. 
\eeq
\end{theorem} 

\begin{proof}
See \cite{Bu18b} or \cite{EvGy15}. The fact that the upper bound in \eqref{bounduniteq} 
is linear in ${\rm max} \{h_* (\alpha ), h_* (\beta )\}$ is also a consequence of the presence of $h_*(\alpha_{n+1})$ 
in the denominator in the logarithm in Theorem \ref{lflog}.
\end{proof}

The next (easy) lemma shows how a lower bound for the norm of $|P(\xi)|$, where $P(X)$ runs through 
the set of integer polynomials of degree at most $n$, 
implies a lower bound for the distance from $\xi$ to algebraic numbers of degree at most $n$. 

\begin{lemma}  \label{transf} 
Let $\xi$ be a complex non-real algebraic number of degree $d$. 
Let $n$ be an integer with $2 \le n \le d-2$. Set $K = \Q(\xi)$. 
If there exist effectively computable positive real numbers $\delta_1, \delta_2, C$, and $X_0$ 
such that, for all integers $x_0, x_1, \ldots , x_n$ with $X = \max\{|x_0|, \ldots , |x_n|\} \ge X_0$, we have
\beq \label{Case_i}
|\Norm_{K / \Q} (x_0 + x_1 \xi + \ldots + x_n \xi^n)| \ge X^{\delta_1} 
\eeq
or
\beq \label{Case_ii}
|x_0 + x_1 \xi + \ldots + x_n \xi^n| \ge X^{-\frac{d}{2} + 1 + \delta_2} \, |\Norm_{K / \Q} (x_0 + x_1 \xi + \ldots + x_n \xi^n)|^{-C}, 
\eeq
then there exist effectively computable, positive $c, \kappa$, depending only on $\xi$, such that 
$$
|\xi - \alpha| > c \, H(\alpha)^{- \frac{d}{2} + \kappa },
$$
for every algebraic number $\alpha$ of degree at most $n$. 
\end{lemma}

\begin{proof}
Let $P(X)$ be an integer polynomial of degree at most $n$ and height at least $X_0$. 
It follows from \eqref{Case_i} that
$$
\prod_\sigma \, | P(\sigma (\xi) )| \ge H(P)^{\delta_1},
$$
where $\sigma$ runs through the complex embeddings of $K$. 
Since $\xi$ is non-real, this shows that there exists $c_1 > 0$, depending on $\xi$, such that 
$$
|P(\xi)| \ge c_1 H(P)^{-\frac{d}{2} + 1 + \frac{\delta_1}{2}}.
$$
Let $\alpha$ be an algebraic number of degree at most $n$ and let $P_\alpha (X)$ denote its minimal 
defining polynomial over the integers. Then, by \cite[Corollary A.1]{BuLiv}, there exists $c_2 > 0$, 
depending on $\xi$, such that
\beq \label{polroot}
|P_\alpha (\xi) | \le c_2 \, H(P) \cdot |\xi - \alpha|, 
\eeq
and the conclusion of the lemma follows from the combination of the last two inequalities.

Without any loss of generality, we can assume that $\delta_1 <  \delta_2 / (2 C)$. 
Assume now that \eqref{Case_ii} holds, but not \eqref{Case_i}. Then,
$$
|P_\alpha (\xi) |  \ge H(P)^{-\frac{d}{2} + 1 + \delta_2} \, H(P)^{-C \delta_1}  \ge H(P)^{-\frac{d}{2} + 1 + \frac{\delta_2}{2}}, 
$$
and we conclude again by \eqref{polroot}.
\end{proof}

\begin{lemma}    \label{xumu}
Let $\bfx$ be a nonzero element of an algebraic number field $K$ and put $M = |\Norm_{K / \Q} (\bfx)|$. 
Then, there exists a unit $u$ in $K$ such that, putting $\mu = \bfx / u$, we have 
$$
C_2^{-1} M^{1/d} \le |\sigma (\mu) | \le C_2 M^{1/d},
$$
for every complex embedding $\sigma$ of $K$, with a constant $C_2 = C_2 (K) \ge 1$.
\end{lemma}

\begin{proof}
This follows from the proof of \cite[Proposition 4.3.12]{EvGy15}. 
\end{proof}

We will apply the following statement to derive Theorems \ref{thw2} to \ref{thw4}. 

\begin{theorem}  \label{thxmu}
Let $\xi$ be a complex non-real algebraic number of degree $d$. 
Let $n$ be an integer with $2 \le n \le d-2$. Set $K = \Q(\xi)$. 
Assume there are effectively computable positive real numbers 
$\kappa_1$ and $h_0$ with the following property. 
For any $\bfx = x_0 + x_1 \xi + \ldots + x_n \xi^n$ in $K$ with $x_0, x_1, \ldots , x_n$ in $\Z$ and
$h(\bfx) > h_0$, upon writing $\bfx = \mu u$ as in Lemma \ref{xumu}, we have
$$
h(\bfx) \le \kappa_1 h_* (\mu).
$$
Then, there exist effectively computable, positive $c, \kappa$, depending only on $\xi$, such that 
$$
|\xi - \alpha| > c \, H(\alpha)^{- \frac{d}{2} + \kappa },
$$
for every algebraic number $\alpha$ of degree at most $n$. 
\end{theorem}

\begin{proof}
Let $\bfx$ be as in the theorem and set $X = \max\{ |x_0|, |x_1|, \ldots , |x_n|\}$. 
There exists an effectively computable constant $C_3 = C_3 (\xi)$ such that 
$$
C_3^{-1} X^{1/d} \le \exp (h(\bfx)) \le C_3 X.
$$
Thus, there are effectively computable constants $X_0 = X_0 (\xi)$ and $C_4 = C_4 (\xi)$ such that 
$h(\bfx) > h_0$ if $X \ge X_0$ and
$$
C_4 |\Norm_{K / \Q} (\bfx)|^{1/d} \ge \exp (h_* (\mu)) \ge \exp ( h(\bfx) / \kappa_1 ) \ge      
(C_3^{-1} X^{1/d})^{1 / \kappa_1}.
$$
This shows that the assumptions of Lemma \ref{transf} are satisfied with suitable 
effectively computable positive real numbers $\delta$ and $X_0$. 
\end{proof}

We end this section with a lemma communicated (with its proof) to us by the referee. 

\begin{lemma}    \label{CM}
Let $\xi$ be a totally complex algebraic number of degree $4$. 
Then, $\Q(\xi)$ is a CM-field if and only if
$1, \xi + \overline{\xi}$ and $\xi \cdot \overline{\xi}$ are linearly dependent over $\Q$.
\end{lemma}

\begin{proof}
Set $\alpha := \xi + \overline{\xi}$ and $\beta := \xi \cdot \overline{\xi}$. 

If $\Q(\xi)$ is a CM-field, then the real numbers $1, \alpha, \beta$ belong to its quadratic real 
subfield, thus they are linearly dependent over $\Q$. 

Conversely, assume there exist integers $a, b, c$ not all zero such that $a  \alpha + b \beta + c = 0$. 
Let $\sigma$ be a field embedding of $\Q(\xi, \overline{\xi})$ into $\C$ that fixes $\xi$. We get     
\begin{align*}
0 & = a \sigma(\alpha) + b \sigma(\beta) + c \\
& = a (\sigma(\alpha) - \alpha) + b (\sigma(\beta) - \beta) \\
& = a (\sigma(\overline{\xi}) - \overline{\xi}) + b ( \xi \sigma(\overline{\xi}) - \xi \overline{\xi}) \\
& = (a + b \xi ) (  \sigma(\overline{\xi}) - \overline{\xi}),
\end{align*}
and $\sigma$ must fix $\overline{\xi}$, since $a$ and $b$ are not both $0$.   
Consequently, the field $\Q(\xi)$ is stable under complex conjugation and thus 
it contains a real quadratic subfield. 
\end{proof}

\section{Proofs}   \label{proofs}

A key ingredient in our proofs is the notion of `matching', introduced in \cite{BajBen}, that we briefly recall.    
Let $a_1, \ldots , a_s$ be non-zero elements in a given number field $K$. 
Let $u_1, \ldots , u_s$ be units in $K$ such that 
\beq \label{eq6}
a_1 u_1 + a_2 u_2 + \ldots + a_s u_s = 0, 
\eeq
and set
$$
\tih = \max\{h_*(u_1), \ldots , h_*(u_s)\}. 
$$
We say that two distinct units $u_\ell, u_k$ in the above equation can be
‘matched’ if we have
$$
h \Bigl( \frac{u_\ell}{u_k} \Bigr) \ll \log_* \tih       
$$
where (as below) the numerical constant
implicit in $\ll$ is effectively computable and depends at most on $K$ and on the heights of the coefficients $a_i$. 
The main strategy for matching units is to show that they are of comparable size at every (infinite) place, 
and deduce that they must then be essentially the same unit up to a multiplicative factor of small height. 
So suppose there exist $1 \le \ell \not= k \le s$ and units 
$u_\ell, u_k$ 
in the equation \eqref{eq6} 
such that for
all complex embedding $\sigma$ of $K$ we have 
\beq \label{eq7}
\bigl| \log | \sigma (u_\ell / u_k) | \bigr| \le \log_*  \tih.
\eeq
Then, writing
$$
a_\ell u_\ell+a_k u_k = a'_k u_k, \quad \hbox{with $a'_k = a_\ell u_\ell /u_k +a_k$,}  
$$
it follows from \eqref{eq7} that
$$
h(a'_k) \ll  \log_* \tih .  
$$
We use effective lower bounds for linear forms in complex logarithms to show that pairs of units can be `matched'. 
Of course, if $K$ is a CM-field, then there is nothing to do, as every unit is matched with its complex conjugate. However, in 
the general case,  complex conjugation has no reason to commute with a Galois embedding.

\subsection{Common preliminaries to the proofs of Theorems \ref{thw2} to \ref{thnf4}.}

Let $n \ge 2$ be an integer and $m$ a nonzero integer. 
Let $\xi$ be a totally complex algebraic number of degree $d \ge n+2$.  
Without any loss of generality, we assume that $\xi$ is an algebraic integer. 
Set $K = \Q(\xi)$ and let $F$ denote the Galois closure of $K$. 
Let $x_0, x_1, \ldots , x_n$ be integers, not all zero, such that 
$$  
\Norm_{K/\Q} (x_0 + x_1\xi + x_2\xi^2 + \ldots + x_n \xi^n) = m. 
$$ 
Set
$$
\bfx = x_0 + x_1\xi + x_2\xi^2 + \ldots + x_n \xi^n.
$$
Let $\iota$ denote the complex conjugation in $\C$. 
Let $\sigma_1, \ldots , \sigma_d$ denote the $d$ embeddings of $K$ into $\C$ 
(recall that $K$ is assumed to be totally complex) 
ordered so that
$$
\sigma_{2j - 1}  = \iota \circ \sigma_{2j}, \quad 1 \le j \le d/2, 
$$
and 
\beq  \label{decr}
|\sigma_1( \bfx  )| =  |\sigma_2( \bfx  )| \ge \cdots \ge |\sigma_{d-1}( \bfx  )| = |\sigma_d( \bfx  )|.
\eeq
For later use, recall that 
$$
d^{-1} \log |\sigma_1( \bfx  )|  \le h(\bfx) \le \log |\sigma_1( \bfx  )|.
$$
and observe that
$$
\sigma_1(\mathbf x) \cdot \sigma_2(\mathbf x) \cdots \sigma_d(\mathbf x) = m. 
$$
We let $u$ and $\mu = \bfx / u$ be given by Lemma \ref{xumu}.

Without any loss of generality, we assume that 
$$
|\sigma_d (\bfx)| < 1,
$$ 
since otherwise we get at once $|\sigma_1 (\bfx)| \le |m|$ and we obtain
$$
d h_*(\mu) \gg  \log_* |m| \ge \log |\sigma_1 (\bfx)| \ge h(\bfx), 
$$
and the desired conclusion follows from Theorem \ref{thxmu}. Here and below, the constants 
implicit in $\gg$ and in $\ll$ are positive, effectively computable and depend at most on $\xi$ and $d$. 

We also assume that      
\beq  \label{xgemu}
h(\bfx) \ge 2 h_*(\mu), 
\eeq
since otherwise the desired conclusion follows from Theorem \ref{thxmu}.     

\begin{lemma}  \label{system}
We keep the above notation. 
Let $\tau_1, \ldots , \tau_{n+2}$ denote $n+2$ distinct complex embeddings of $K$. Then, the system of $n+1$ 
linear equations
$$
\tau_1 (\xi^j) + \sum_{i=2}^{n+2} a_i \tau_i (\xi^j) = 0, \quad \hbox{$0 \le j \le n$}, 
$$
has a unique solution $(a_2, \ldots , a_{n+2})$ in $F^{n+1}$. 
For $i=2, \ldots , n+2$, we have $a_i \not= 0$ and $h_* (a_i) \le C_5$, for a positive, effectively computable constant     
$C_5$ depending only on $\xi$.    
\end{lemma}

\begin{proof}
This is a Cramer system and $a_2, \ldots , a_{n+2}$ are quotients of 
nonzero van der Monde determinants. Their heights are bounded by some constant depending   
on $\tau_1 (\xi), \ldots , \tau_{n+2} (\xi)$. Since there are only finitely many choices 
for $\tau_1, \ldots , \tau_{n+2}$, the heights of $a_2, \ldots , a_{n+2}$ are bounded by some constant 
depending only on $\xi$.   
\end{proof}

\begin{lemma} \label{boundsigma1} 
We keep the above notation. 
We have $|\sigma_1 (\bfx) | \le C_6 |\sigma_{d-n} (\bfx)|$, where 
$C_6 = d \exp ( C_5  [F : \Q] )$. 
\end{lemma}

\begin{proof}
It follows from Lemma \ref{system} applied to the $n+2$ embeddings 
$\sigma_1, \sigma_{d-n}, \ldots , \sigma_d$ that there are nonzero $a_0, \ldots , a_n$ in $F$ of logarithmic height 
at most $C_5$ such that 
$$
\sigma_1 (\bfx) + \sum_{j=0}^n a_j \sigma_{d-j} (\bfx) = 0. 
$$
We conclude by using \eqref{decr}. 
\end{proof}

The following lemma, whose formulation was kindly provided to us by the referee, allows us to clarify  
the proofs of Theorems \ref{thw3} to \ref{thnf4}.    

\begin{lemma} \label{Galois}
Assume that $F = K = \Q(\xi)$ is a Galois extension of $\Q$, so that its Galois group 
is $G = \{\sigma_1, \ldots , \sigma_d\}$. Then, the sets $\{\sigma \sigma_{d-1}, \sigma \sigma_d\}$ with 
$\sigma$ in $G$ form a partition of $G$.    
Furthermore, we have
$$
h \left( \frac{(\sigma \sigma_{d-1}) (u)}{(\sigma \sigma_{d}) (u)} \right) \le 2 \log C_2 + \log \frac{|\sigma_{1} (\bfx)| }{ |\sigma_{d-2} (\bfx)|}, 
$$
for every $\sigma$ in $G$.     
\end{lemma}

\begin{proof}
Define 
$$
\iota' = \sigma_d^{-1} \iota \sigma_d 
\quad \hbox{and}
\quad
T = \{\sigma_{d-1}, \sigma_d \} = \langle \iota \rangle   \sigma_d = \sigma_d \langle \iota' \rangle,
$$
where $ \langle \iota \rangle = \{1, \iota\}$ and $ \langle \iota' \rangle = \{1, \iota'\}$. 
Since the left translates $\sigma T$ of $T$ with $\sigma$ in $G$ are the left translates of the subgroup $ \langle \iota' \rangle$, 
they form a partition of $G$. 
Noticing that $|\sigma (\bfx) | \le |\sigma_{1} (\bfx)| \cdot |\tau(\bfx)| / |\sigma_{d-2} (\bfx)|$ when $\sigma$ and $\tau$ 
are both in $T$ or both in $G \setminus T$, 
we derive that 
$$
|(\sigma \sigma_{d-1}) (\bfx) | \le \frac{|\sigma_{1} (\bfx)| }{ |\sigma_{d-2} (\bfx)|} \cdot |(\sigma \sigma_{d}) (\bfx) |, \quad
\hbox{for any $\sigma$ in $G$.} 
$$
Consequently, for each $\sigma$ in $G$, we get 
$$
\left| \frac{ (\sigma \sigma_{d-1}) (u)}{ (\sigma \sigma_{d}) (u)} \right|
= \left| \frac{ (\sigma \sigma_{d-1}) (\bfx)}{ (\sigma \sigma_{d}) (\bfx) } \cdot 
\frac{ (\sigma \sigma_{d}) (\mu)}{ (\sigma \sigma_{d-1}) (\mu) } \right| 
\le C_2^2 \frac{|\sigma_{1} (\bfx)| }{ |\sigma_{d-2} (\bfx)|}, 
$$
and all the Galois conjugates of the unit $(\sigma \sigma_{d-1}) (u) / (\sigma \sigma_{d}) (u)$       
are of absolute value at most $C_2^2 |\sigma_{1} (\bfx)| / |\sigma_{d-2} (\bfx)|$.    
This implies the result. 
\end{proof}

\subsection{Proofs of Theorems \ref{thw2} to \ref{thnf4}.}

\begin{proof}[Proof of Theorem \ref{thw2} when $d=4$] 
We have $n=2$.     
Let $\eta$ be a generator of the group of units, modulo torsion, of the totally complex quartic field $K$.    
We claim that $\sigma_1 (\eta) / \sigma_2 (\eta)$ is not a root of unity of $F$. 
Indeed, if $k$ were its order, then $(\sigma_1 (\eta))^k$ would be a unit of infinite order in the 
field $\sigma_1 (K) \cap \R$, thus a real quadratic unit and $K$ would be a CM-field. 

Write $u = \rho \eta^b$, for some root of unity $\rho$ in $K$ and some integer $b$. 

It follows from Lemma \ref{system} that there are nonzero $a_2, a_3, a_4$ in $F$ such that 
$$
\sigma_1 (\bfx) + \sum_{i=2}^4 a_i \sigma_{i} (\bfx) = 0,
$$
Upon dividing by $\sigma_1 (\bfx)$
and recalling that $\bfx = \mu u$, we get 
$$
|1 - \alpha v| \le \sum_{i=3}^4 \Bigl| \frac{a_i \sigma_{i} (\bfx)}{\sigma_1 (\bfx)} \Bigr|, 
$$
with
$$ 
\alpha = - \frac{a_2 \sigma_{2} (\mu)}{\sigma_1 (\mu)}, \quad 
v = \frac{\sigma_{2} (u)}{\sigma_1 (u)} 
= \frac{\sigma_{2} (\rho)}{\sigma_1 (\rho) }  \Bigl( \frac{\sigma_{2} (\eta)}{\sigma_1 (\eta)}  \Bigr)^b. 
$$
Setting $\kappa := h( \sigma_1 (\eta) / \sigma_2 (\eta) ) / h(\eta)$, we have 
$\kappa > 0$, since $\sigma_2 (\eta) / \sigma_1 (\eta)$ is not a root of unity, and $\kappa \le 2$.    
Furthermore,    
\beq \label{etaxmu}
h(v) = |b| h(\sigma_1 (\eta) / \sigma_2 (\eta)) = \kappa  |b| h(\eta) = \kappa h(u) \ge \kappa (h(\bfx) - h(\mu)).
\eeq
If $h(v)  \le h(\alpha)$, then, using that $h(\alpha) \ll h_* (\mu)$, we derive from \eqref{etaxmu} that $h(\bfx) \ll h_* (\mu)$ and 
we conclude by applying Theorem \ref{thxmu}. 

If $h(v) > h(\alpha)$, then $|1 - \alpha v|$ is nonzero and,
since $|\sigma_{3} (\bfx)| = |\sigma_4 (\bfx) | < 1$, we get $|1 - \alpha v| \ll |\sigma_1 (\bfx) |^{-1}$, thus 
$\log |1 - \alpha v| \le - h(\bfx) / 2$ if $h(\bfx)$ is large enough. 
Note that $v$ is a unit or a root of unity.    
It then follows from Theorem \ref{uniteq} that    
\beq    \label{minmaj}
\frac{h(\bfx)}{2} \le - \log |1 - \alpha v| \ll h_* (\alpha) \log_* \frac{h_*(v)}{h_* (\alpha)}.
\eeq

Furthermore, by \eqref{xgemu},  we get    
$$
h(v) \ll h(u) \ll (h(\bfx) +  h_* (\mu) ) \ll h(\bfx). 
$$
We derive from \eqref{minmaj} that 
$$
h(\bfx) \ll h_* (\alpha) \ll h_* (\mu),
$$
and we conclude by applying Theorem \ref{thxmu}. 
\end{proof}


\begin{proof}[Proof of Theorem \ref{thw2} when $d \ge 6$] 
We have $n=2$.      
It follows from Lemma \ref{system}  that there are nonzero $a_0, a_1, a_2, b_0, b_1, b_2$ in $F$ 
such that 
$$
\sigma_1 (\xi^j) + \sum_{i=0}^2 a_i \sigma_{d-i} (\xi^j) = 0 
\quad
\hbox{and}
\quad
\sigma_2 (\xi^j) + \sum_{i=0}^2 b_i \sigma_{d-i} (\xi^j) = 0, 
\quad
\hbox{for $j = 0, 1, 2$}.
$$
This implies 
$$
b_0 \sigma_1 (\xi^j) - a_0 \sigma_2 (\xi^j) + \sum_{i=1}^2 (a_i b_0 - a_0 b_i) \sigma_{d-i} (\xi^j)= 0, 
\quad
\hbox{for $j = 0, 1, 2$}.
$$
Consequently we have $a_1 b_0 \not= a_0 b_1$. Thus, $a_1 \sigma_{d-1} (\bfx) + a_0 \sigma_d(\bfx)$ 
and $b_1 \sigma_{d-1} (\bfx) + b_0 \sigma_d(\bfx)$ cannot be both $0$. 
By permuting $\sigma_1$ and $\sigma_2$ if necessary, 
we assume that $a_1 \sigma_{d-1} (\bfx) + a_0 \sigma_d(\bfx)$ is nonzero. 
By linearity, we have
$$
\sigma_1 (\bfx) + \sum_{i=0}^2 a_i \sigma_{d-i} (\bfx) = 0,
$$
thus $\sigma_1 (\bfx) + a_2 \sigma_{d-2} (\bfx)$ is nonzero. Consequently, by dividing by $\sigma_1 (\bfx)$
and recalling that $\bfx = \mu u$, we get 
$$
0 < |1 - \alpha v| \le \sum_{i=0}^1 \Bigl| \frac{a_i \sigma_{d-i} (\bfx)}{\sigma_1 (\bfx)} \Bigr|, 
\quad \hbox{with}
\quad 
\alpha = - \frac{a_2 \sigma_{d-2} (\mu)}{\sigma_1 (\mu)}, v = \frac{\sigma_{d-2} (u)}{\sigma_1 (u)}. 
$$
Since $|\sigma_{d-1} (\bfx)| = |\sigma_d (\bfx) | < 1$, this gives $|1 - \alpha v| \ll |\sigma_1 (\bfx) |^{-1}$, thus 
$\log |1 - \alpha v| \le - h(\bfx) / 2$ if $h(\bfx)$ is large enough. By Theorem \ref{uniteq}, this gives
$$
\frac{h(\bfx)}{2} \le - \log |1 - \alpha v| \ll  h_* (\alpha) \log_* \frac{h_*(v)}{h_* (\alpha)}.
$$
We conclude as in the proof of the case $d=4$. 
\end{proof}


\begin{proof}[Proof of Theorem \ref{thw3}] 
We have $n=3$.      
Set $T = \{\sigma_{d-1}, \sigma_d\}$. Since $G$ has order at least $6$, it has three disjoint left translates of $T$
of the form
$$
\{\sigma, \tau\}, \quad \{\nu, \sigma_{d-2}\}, \quad T = \{\sigma_{d-1}, \sigma_d\},
$$
where $\sigma, \tau, \nu$ are distinct elements in $\{\sigma_1, \ldots , \sigma_{d-3}\}$. 
It follows from Lemma \ref{system}  that there are nonzero $a_0, a_1, a_2, a_3, b_0, b_1, b_2, b_3$ 
in $K = F$ 
such that 
$$
\sigma (\xi^j) + a_3 \nu (\xi^j) + \sum_{i=0}^2 a_i \sigma_{d-i} (\xi^j) = 0 
\quad
\hbox{and}
\quad
\tau (\xi^j) + b_3 \nu (\xi^j) +  \sum_{i=0}^2 b_i \sigma_{d-i} (\xi^j) = 0, 
$$
for $j = 0, 1, 2, 3$. 
This implies 
$$
b_0 \sigma (\xi^j) - a_0 \tau (\xi^j) + (a_3 b_0 - a_0 b_3) \nu (\xi^j) + \sum_{i=1}^2 (a_i b_0 - a_0 b_i) \sigma_{d-i} (\xi^j)= 0,     
\quad
\hbox{for $j = 0, 1, 2, 3$}.     
$$
Consequently we have $a_1 b_0 \not= a_0 b_1$. Thus, $a_1 \sigma_{d-1} (\bfx) + a_0 \sigma_d(\bfx)$     
and $b_1 \sigma_{d-1} (\bfx) + b_0 \sigma_d(\bfx)$ cannot be both $0$.      
By permuting $\sigma$ and $\tau$ if necessary,      
we may assume that $a_1 \sigma_{d-1} (\bfx) + a_0 \sigma_d(\bfx)$ is nonzero.    
Then we get
$$
0 < |\sigma (\bfx) + a_3 \nu(\bfx) + a_2 \sigma_{d-2} (\bfx)| = |a_1 \sigma_{d-1} (\bfx) + a_0 \sigma_d(\bfx)|.
$$
By dividing by $\sigma (\bfx)$ and noting that 
$|\sigma (\bfx)| \ge |\sigma_{d-3} (\bfx)| \gg |\sigma_1 (\bfx)|$, by Lemma \ref{boundsigma1}, 
we get 
$$
0 <  |1 - \alpha v| \le \sum_{i=0}^1 \Bigl| \frac{a_i \sigma_{d-i} (\bfx)}{\sigma (\bfx)} \Bigr| \ll \frac{1}{|\sigma_1 (\bfx)|}, 
$$
with
$$
\alpha = - \frac{a_3 \nu(\mu)}{\sigma (\mu)} \theta - \frac{a_2 \sigma_{d-2} (\mu)}{\sigma (\mu)}, \quad
\theta = \frac{\nu(u)}{\sigma_{d-2} (u)}
\quad v = \frac{\sigma_{d-2} (u)}{\sigma (u)}. 
$$
Note that since $|\sigma_{d-2}(\bfx)| = |\sigma_{d-3}(\bfx)| \gg |\sigma_1(\bfx)|$, 
we get from Lemma \ref{Galois} that $h(\theta)\ll 1$.
By \eqref{xgemu} and assuming $h(\bfx)$ large enough, we proceed as in the above proofs to get    
$$
\frac{h(\bfx)}{2} \le - \log |1 - \alpha v| \ll  h_* (\alpha) \log_*  \frac{h_*(v)}{h_* (\alpha)}.
$$
As $h(v) \le 2 h(u) \le 2 (h(\bfx) + h_*(\mu)) \ll h(\bfx)$, this gives
$$
h(\bfx) \ll h_* (\alpha) 
\ll h_* (\mu),    
$$
and we conclude by applying Theorem \ref{thxmu}.  
\end{proof}

\begin{proof}[Proof of Theorems \ref{thw4} and \ref{thnf4}.]
We have $n=4$ and $d \ge 8$.       
As in the proof of Theorem \ref{thw3}, set $T = \{\sigma_{d-1},\sigma_d\}$. We distinguish three cases. 
Cases 1 and 2 are similar: we assume that certain subsums in unit equations involving conjugates of $\sigma_1(\bfx)$ do not vanish. 
Case 3 then deals with the possibility of vanishing subsums. 
\medskip

\noindent 

\textbf{Case 1:} Suppose $\{\sigma_{d-3},\sigma_{d-2}\}$ is not a left translate of $T$. In this case consider the left-translates of $T$ given by
\[
\{ \sigma, \sigma_{d-3} \}, \quad \{ \nu , \sigma_{d-2} \},
\]
which are disjoint from each other (and from $T$) by assumption. 
Since $\sigma_{d-3} = \iota \circ \sigma_{d-2}$, this implies $\sigma = \iota \circ \nu$.     
We have from Lemma \ref{system} that there are non-zero $a, a_0, a_1, a_2, a_3$ in $K = F$ such that 
\[
\sigma(\xi^j) + a \nu(\xi^j)  + \sum_{i=0}^3 a_i \sigma_{d-i} (\xi^j) = 0
\]
for $j = 0,1,2,3,4$. Consequently, 
\begin{equation}\label{beforematch}
\sigma(\bfx) + a \nu(\bfx) + \sum_{i=0}^3 a_i \sigma_{d-i} (\bfx) = 0.
\end{equation}
We assume in Case 1 that none of the subsums
\[
\sigma(\bfx) + a \nu(\bfx) , \;\, \sigma(\bfx) +  a_3\sigma_{d-3}(\bfx),  \;\, a \nu(\bfx) +  a_2\sigma_{d-2}(\bfx),    
\;\, a_1\sigma_{d-1}(\bfx) +  a_0\sigma_{d}(\bfx)
\]
equals zero, delaying the case of vanishing subsums to Case 3. Since $\sigma(\bfx) + a \nu(\bfx) \neq 0$, we have
\beq  \label{eqcaseI}
0 < |\sigma(\bfx) + a \nu(\bfx) | = \left|  \sum_{i=0}^3 a_i \sigma_{d-i} (\bfx) \right|.
\eeq 
Theorem \ref{uniteq} then gives that   
\[
-\log |1 - \alpha v| \ll h_*(\alpha)\log_* \left( \frac{h_*(v)}{h_*(\alpha)} \right) \ll h_*(\alpha)\log_* \left( \frac{h_*(u)}{h_*(\alpha)} \right), 
\]
where
\[
\alpha =  \frac{-a\nu(\mu)}{\sigma(\mu)} \quad \text{and} \quad v = \frac{\nu(u)}{\sigma(u)}.
\]
Set
\[
\Omega = h_*(\alpha)\log_* \left( \frac{h_*(u)}{h_*(\alpha)} \right).
\]
Thus, we have
\[
\log\frac{|\sigma(\bfx)|}  {|\sigma(\mathbf x) + a\nu(\mathbf x)|} \ll \Omega.
\]
Since $\sigma = \sigma_h$ for an integer $h$ with $1 \le h \le d-4$, it follows from Lemma \ref{boundsigma1} that    
$|\sigma( \bfx)| \gg |\sigma_1(\bfx)|$. 
By \eqref{eqcaseI}, we have $|\sigma(\mathbf x) + a\nu(\mathbf x)| \ll  |\sigma_{d-3}(\bfx)|$. Consequently, the above argument yields
\[
\log\frac{|\sigma_1(\mathbf x)|}{|\sigma_{d-2}(\bfx)|} = \log\frac{|\sigma_1(\bfx)|}{|\sigma_{d-3}(\bfx)|} 
\ll \log\frac{|\sigma (\bfx)|}{|\sigma(\bfx) + a \nu(\bfx)|} \ll \Omega .     
\]

The next step is to apply Lemma \ref{Galois}, whence for any $\tau$ in $G$ we have

\[
h\left(  \frac{ (\tau\sigma_{d-1}) ( \bfx)}{(\tau\sigma_{d}) ( \bfx)} \right) 
\ll h\left(  \frac{(\tau\sigma_{d-1})(u)}{(\tau\sigma_{d})(u)} \right) +  h_*(\mu)\ll  \max\{\Omega, h_*(\mu)\}.
\]
In particular, we get
\beq  \label{Om}
h\left( \frac{\sigma(\bfx)}{\sigma_{d-3}(\bfx)} \right),\;  h\left( \frac{\nu(\bfx)}{\sigma_{d-2}(\bfx)} \right), \; h\left( \frac{\sigma_{d-1}(\bfx)}{\sigma_d(\bfx)} \right) \; \ll \max\{\Omega, h_*(\mu)\}.
\eeq 
Note the upper bound $\Omega$ here depends on $h(\bfx)$. We use this to rewrite \eqref{beforematch} as
\begin{equation}\label{aftermatch}
	\beta \sigma(u) + \gamma \sigma_{d-2}(u) + \eta \sigma_d(u)  = 0, 
\end{equation}
where
\begin{align*}
	\beta = \sigma(\mu) + a_3 \sigma(\mu)&\frac{\sigma_{d-3}(\bfx)}{\sigma(\bfx)}, \qquad    
	\gamma = a \sigma_{d-2}(\mu)\frac{\nu(\bfx)}{\sigma_{d-2}(\bfx)} + a_2\sigma_{d-2}(\mu), \\    
	\eta &= a_1\sigma_d (\mu)\frac{\sigma_{d-1}(\bfx)}{\sigma_d(\bfx)} + a_0\sigma_d(\mu), 
\end{align*}
and each of $\beta, \gamma, \eta$ is nonzero by assumption. We deduce from \eqref{Om} that
\[
\max\{h_*(\beta),h_*(\gamma),h_*(\eta)\} \ll \max\{\Omega, h_*(\mu)\}. 
\]
We normalize the unit equation \eqref{aftermatch} by dividing each of its terms by $\eta \sigma_d(u)$. 
Then, by applying Theorem \ref{threeuniteq}, we obtain 
\[
\max\left\{ h\left(  \frac{\sigma(u)}{\sigma_d(u)} \right), h\left(  \frac{\sigma_{d-2}(u)}{\sigma_d(u)} \right)  \right\} \ll \max\{\Omega, h_*(\mu)\}.
\]
It follows from Lemma \ref{xumu} that
$$
C_2^{-2} \frac{|\sigma_i (u)|}{|\sigma_j(u)|} \le \frac{|\sigma_i (\bfx)|}{|\sigma_j(\bfx)|}  \le C_2^2 \frac{|\sigma_i (u)|}{|\sigma_j(u)|}, 
\quad 1 \le i, j \le d. 
$$
Consequently, by \eqref{decr} and recalling that $|\sigma_1(\bfx)|  \ll |\sigma(\bfx)|$, we get  
$$
|\sigma_d(u)| \ll \min_{1\le i\le d} \, |\sigma_i(u)|  \ll \max_{1\le i\le d} \, |\sigma_i(u)| \ll |\sigma_1(u)| \ll  |\sigma(u)|. 
$$
From these inequalities (and the fact that $u$ is a unit) we have
\begin{align*}
h(u) = \tfrac12 \big(h(u) + h(u^{-1})\big) &\le \tfrac12 \Bigl( \log\max_{1\le i\le d}\, |\sigma_i(u)|  - \log\min_{1\le i\le d}\, |\sigma_i(u)| \Bigr) \\
&\ll \log_* \frac{|\sigma(u)|}{|\sigma_d(u)|} \le  d\cdot h\left(  \frac{\sigma(u)}{\sigma_d(u)} \right).
\end{align*}
This gives us
\[
h(u) \ll \max\{\Omega, h_*(\mu)\}.
\]
If $\Omega < h_*(\mu)$, then $h(u) \ll h_*(\mu)$.
Otherwise, we have
\[
h(u) \ll  \Omega  \ll h_*(\alpha)\log_* \left( \frac{h_*(u)}{h_*(\alpha)} \right), 
\]
thus $h(u) \ll h_*(\alpha)$. Consequently, in both cases, we deduce 
that 
\[
h(\bfx) \le h(\mu) + h(u) \ll h_*(\mu) + h_*(\alpha)  \ll h_*(\mu), 
\]
and we conclude by applying Theorem \ref{thxmu}. 

\medskip

\noindent 
\textbf{Case 2: } Now suppose $\{\sigma_{d-3},\sigma_{d-2}\}$ is itself a left translate of $T$. Letting $\{\sigma, \sigma_{d-4}\}$ be another left translate of $T$ and using Lemma \ref{system} we get an equation
\begin{equation}\label{beforecase2}
\sigma(\bfx) + \sum_{i=0}^4 b_i \sigma_{d-i}(\bfx)	= 0.
\end{equation}
for suitable nonzero $b_0,\ldots,b_4$ in $K$. We assume none of the subsums 
\[
\sigma(\bfx) +  b_4\sigma_{d-4}(\bfx), \quad b_3\sigma_{d-3}(\bfx) +  b_2\sigma_{d-2}(\bfx), \quad b_1\sigma_{d-1}(\bfx) +  b_0\sigma_{d}(\bfx)
\]
vanishes. In this case consider
\[
0 < |\sigma(\bfx) +  b_4\sigma_{d-4}(\bfx)| = \left| \sum_{i=0}^3 b_i \sigma_{d-i}(\bfx) \right| \ll  |\sigma_{d-3}(\bfx)|.
\]
We continue from here as in Case 1, using Lemma \ref{uniteq} to get
\[
-\log|1-\alpha v| \ll h_*(\alpha)\log_*\left(\frac{h_*(u)}{h_*(\alpha)}\right),
\]
with
\[
\alpha = \frac{-b_4\sigma_{d-4}(\mu)}{\sigma(\mu)} \quad \text{and} \quad v = \frac{\sigma_{d-4}(u)}{\sigma(u)}.
\]
As in Case 1 we deduce from Lemma \ref{boundsigma1} that $|\sigma (\bfx) | \gg |\sigma_1 (\bfx)|$, and we find that 
\[
\log\frac{|\sigma_1(\mathbf x)|}{|\sigma_{d-2}(\bfx)|} \ll \Omega, 
\]
where
\[
\Omega = h_*(\alpha)\log_*\left( \frac{h_*(u)}{h_*(\alpha)} \right).
\]
Now we can rewrite \eqref{beforecase2} as
\begin{equation*}
	\beta \sigma(u) + \gamma \sigma_{d-2}(u) + \eta \sigma_d(u)  = 0, 
\end{equation*}
where
\begin{align*}
	\beta = \sigma(\mu) + b_4\sigma(\mu)&\frac{\sigma_{d-4}(\bfx)}{\sigma(\bfx)}, \qquad
	\gamma = b_3\sigma_{d-2}(\mu)\frac{\sigma_{d-3}(\bfx)}{\sigma_{d-2}(\bfx)} + b_2\sigma_{d-2}(\mu), \\
	\eta &= b_1\sigma_d (\mu)\frac{\sigma_{d-1}(\bfx)}{\sigma_d(\bfx)} + b_0\sigma_d(\mu).
\end{align*}
We complete the argument as in Case 1.

\medskip

\noindent 
\textbf{Case 3: } It remains for us to deal with the eventuality of vanishing subsums. 
The first possibly vanishing subsum is $\sigma(\bfx) + a \nu(\bfx)$ from Case 1, involving some $a$ in $K$ which, by 
Lemma \ref{system}, is nonzero and satisfies $h_*(a) \le C_5$. If this subsum vanishes, then we have
$$
(\iota \circ \nu) (\bfx) = - a \nu (\bfx), 
$$
since $\sigma = \iota \circ \nu$. All the other possibly vanishing subsums take the form
$$
a' (\tau \sigma_{d-1}) (\bfx) + a'' (\tau \sigma_d) (\bfx),
$$
for some $\tau$ in the Galois group $G$ of $K$ and some $a', a''$ in $K$ which, by 
Lemma \ref{system}, are nonzero and satisfy $h_*(a'/a'') \le C_5$. If such a subsum vanishes, then we have
$$
(\iota \circ \sigma_d) (\bfx) = - \tau^{-1} (a'' / a') \sigma_d (\bfx), 
$$
since $\sigma_{d-1} = \iota \circ \sigma_d$. Consequently, assuming that a subsum vanishes, we have in all cases
$$
(\iota \circ \tau) (\bfx) =  b \tau (\bfx), 
$$
for some $\tau$ in $G$ and some nonzero $b$ in $K$ with $h_* (b) \le C_5$. 
Writing $\otau = \iota \circ \tau$ and recalling that $\bfx = \mu u$, this yields
$$
\otau(u) = \eps \tau(u), \quad \hbox{with $\eps = b \tau(\mu) / \otau(\mu)$}.
$$
Thus, $\eps$ is a unit of $K$ with $h_*(\eps) \ll h_*(\mu)$. 

Let $U$ denote the group of units of $K$, let $U_{\rm tor}$ denote its torsion part (that is, the group of roots of unity in $K$), and let 
$\{\eta_1,\ldots,\eta_r\}$ be a fundamental system of units in $K$. Any $v$ in $U$ factors as 
$$
v = \rho \eta_1^{t_1} \ldots \eta_r^{t_r},
$$
for a unique root of unity $\rho$ in $U_{\rm tor}$ and a unique integer vector 
$\undt = (t_1, \ldots , t_r)$ in $\Z^r$. Then, we have
$$
\otau (v) / \tau (v) = \rho' \eta_1^{\ell_1 (\undt)} \ldots \eta_r^{\ell_r(\undt)},
$$
for some $\Z$-linear forms $\ell_1, \ldots , \ell_r$ from $\Z^r$ to $\Z$ depending only on $\tau$, and some 
$\rho'$ in $U_{\rm tor}$ that varies with $v$. 
Let $\cL : \Z^r \to \Z^r$ be the $\Z$-linear map given by
$$
\cL(\undt) = (\ell_1(\undt), \ldots , \ell_r(\undt)), \quad \hbox{for $\undt$ in $\Z^r$.} 
$$

Write the units $u$ and $\eps$ as
$$
u  = \rho_u \eta_1^{b_1} \ldots \eta_r^{b_r}, \quad
\eps  = \rho_\eps \eta_1^{k_1} \ldots \eta_r^{k_r},
$$
with $\rho_u, \rho_\eps$ in $U_{\rm tor}$ and 
$\undb = (b_1, \ldots , b_r)$, $\undk = (k_1, \ldots , k_r)$ in $\Z^r$. 
Then, the relation $\otau(u) = \eps \tau(u)$ implies that $\cL(\undb) = \undk$. If $\cL$ is
injective, this uniquely determines $\undb$ and we obtain
$$
h_* (u) \asymp \| \undb \| \ll \| \undk \| \asymp h_* (\eps) \ll h_*(\mu),
$$
where $\| \cdot \|$ denotes the maximum norm and, here and below, $\asymp$ means that 
both inequalities $\ll$ and $\gg$ hold. 
We get $h(\bfx) \ll h_*(\mu)$ and we conclude by applying Theorem \ref{thxmu}.  

In general, we can write $\undk = \cL(\undb) = \cL (\undt)$ for some $\undt = (t_1, \ldots , t_r)$ in $\Z^r$ satisfying 
$\| \undt \| \le c(\cL) \| \undk \|$, for a positive, effectively computable 
$c(\cL)$ that depends only on $\cL$, thus only on $\tau$. 
Putting $u_1 = \eta_1^{t_1} \ldots \eta_r^{t_r}$, we thus have
$$
h_* (u_1) \asymp \| \undt \| \ll \| \undk \| \ll h_* (\mu)
$$
and
$$
\otau (u) / \tau (u) = \rho'' \otau (u_1) / \tau (u_1),
$$
for some $\rho''$ in $U_{\rm tor}$. 
Observe that $\rho'' = \otau(u/u_1) / \tau(u/u_1)$, where $u/u_1$ is in $U$. Since $U_{\rm tor}$ is a finite set, we may write 
$\rho'' = \otau(u_2) / \tau (u_2)$ for some $u_2$ from a finite, effectively computable subset of $U$. 
Putting $u_0 = u_1 u_2$ with that choice of $u_2$, we find that 
$$
h_*(u_0) \le h_*(u_1) + h_* (u_2) \ll h_* (\mu) + 1 \ll h_* (\mu)
$$
and
$$
\otau(u / u_0) = \tau (u / u_0). 
$$
Identifying $\varphi = \tau^{-1}\otau$ as an element of $\Gal(K/\Q)$, we see that $u/u_0$ 
must lie in the fixed proper subfield $K^{\varphi}$ of $K$. 
Observe that an element $a$ of $K$ belongs to $K^\varphi$ if and only if $\otau(a) = \tau (a)$, that is, 
if and only if $\tau(a)$ is real. Thus, $\tau$ is a real embedding of $K^\varphi$. 

Note that if $K^\varphi = \Q$, then $u/u_0 = \pm 1$, hence $h_* (u) = h_* (u_0) \ll h(\mu)$, and we are done. 

Write $v = u/u_0$ and set $\M$ to be the $\Z$-module
\[
\M = \{x_0 + x_1\xi + x_2\xi^2 + x_3\xi^3 + x_4\xi^4 \; : \; x_0,\ldots,x_4 \in \Z\}.
\]
Since $\mu u_0 v = \bfx$ is in $\M$, we have  
\[
v \in K^{\varphi}\cap (\mu u_0)^{-1}\M.
\]
To estimate the rank of the $\Z$-module $K^{\varphi}\cap (\mu u_0)^{-1}\M$, we use the 
following lemma, which follows from \cite[Lemma 6.1]{BuEv09}.

\begin{lemma}[Bugeaud-Evertse]  \label{LemmaBuEv}   
Let $n$ be a positive integer. Let $\eta$ be a complex non-real algebraic number with $\deg \eta > 2n - 2$.  
Suppose that $1, \eta + \overline\eta, \eta \cdot\overline\eta$ are linearly independent over $\Q$. Then,
for any nonzero $\lambda$ in $\C$, the $\Z$-module
\[
\R \cap \{ \lambda f(\eta) : f(X) \in \Z[X],  \deg f \le n \}
\]
has rank at most $(n+1)/2$.
\end{lemma}

For $n=4$ and each $i=1, \ldots , d$, the number $\eta = \sigma_i (\xi)$ satisfies the assumption of 
Lemma \ref{LemmaBuEv}. Consequently, for any nonzero $\lambda$ in $\C$, the $\Z$-module
$$
\R \cap \lambda \sigma_i (\M)
$$
has rank $\le 5/2$, thus in fact rank at most $2$. 
We choose $i$ so that $\sigma_i = \tau$ and set $\lambda = \tau (\mu u_0)^{-1}$. 
Since $\tau(K^\varphi)$ is included in $\R$, we deduce that the image under $\tau$ of 
$$
K^\varphi\cap(\mu u_0)^{-1}\M
$$
has rank at most $2$ and thus this $\Z$-module also has rank at most $2$.

Thus, there exist $\alpha_1,\alpha_2$ in $K^\varphi$ such that
\begin{equation}\label{alpha12}
	K^\varphi \cap (\mu u_0)^{-1}\mathfrak M \; \subset \; \{ y_1\alpha_1 + y_2\alpha_2 \; : \; y_1,y_2 \in \Z   \}.
\end{equation}
Let $\undx = (x_0, x_1, \ldots , x_4)$ be an integer tuple and assume that 
\begin{align*}
	(\mu u_0)^{-1}(x_0 + x_1\xi &+ x_2\xi^2 + x_3 \xi^3 + x_4 \xi^4) \\
	&- \varphi\left((\mu u_0)^{-1}(x_0 + x_1\xi + x_2\xi^2 + x_3 \xi^3 + x_4 \xi^4)\right) = 0.
\end{align*}
Expressing the left hand side in the basis $1, \xi, \ldots , \xi^{d-1}$ of $\Q(\xi)$, there are 
linear forms $\ell_0, \ldots , \ell_{d-1}$ with rational coefficients of heights 
$\ll h_*(\mu)$ such that 
$$
\ell_0 (\undx) + \ell_1 (\undx) \xi + \ldots + \ell_{d-1} (\undx) \xi^{d-1} = 0.
$$
Consequently, $\undx$ is a solution of the system of equalities
\beq \label{syst}
\ell_0 (\undy) = \ell_1 (\undy) = \ldots =  \ell_{d-1} (\undy) = 0, \quad \hbox{in $\undy \in \Z^5$}. 
\eeq
It follows from \eqref{alpha12} that the set of solutions to \eqref{syst} is a $\Z$-module of rank at most $2$. 
By our observation on the heights of the coefficients of the linear forms $\ell_0, \ldots , \ell_{d-1}$, this 
set has a $\Z$-basis of vectors whose coefficients are rationals with heights 
$\ll h_*(\mu)$. 
The upshot is that there exist $\alpha_1,\alpha_2$ as in the description of \eqref{alpha12} that moreover satisfy
\[
\max\{h(\alpha_1),h(\alpha_2)\} \ll h_*(\mu).
\]
Recall we have $\mathbf x = \mu u_0 v$.  Setting $L = \Q(\alpha_1,\alpha_2)$, we are now interested in the units $v$ in $L$ that satisfy
\[
v \in L\cap (\mu u_0)^{-1}\M.
\]
If $L\cap (\mu u_0)^{-1}\M$ has rank 1 then $L \cap (\mu u_0)^{-1}\mathfrak M = \Z\alpha_3$ 
for some fixed $\alpha_3$ in $L$. 
We have $h(\alpha_3)\ll h_*(\mu)$. If $y_1\alpha_3$ is a unit for some integer $y_1$, 
then $y_1$ must divide the denominator of $\Norm(\alpha_3)$, 
and so in particular $\log |y_1| \ll h_*(\mu)$ as well. All in all, we get 
\[
h(\mathbf x) = h(\mu u_0 \cdot y_1\alpha_3) \ll h_*(\mu), 
\]
whence one can apply Theorem \ref{thxmu}. Thus we may now assume
\[
\mathrm{rank} \bigl( L\cap (\mu u_0)^{-1}\M \bigr) = 2.
\]

First, suppose the subfield $\Q(\alpha_1/\alpha_2)$ of $L$ has degree 3 or higher over $\Q$. 
This case can be dealt with analogous to the effective resolution of Thue equations. 
Choose three embeddings $\tau_1,\tau_2,\tau_3$ of $L$ such that 
\[
\tau_j(\alpha_1/\alpha_2) \neq \tau_k(\alpha_1/\alpha_2), \qquad 1 \le j < k \le 3,
\]
and define
\[
c_i = \det\begin{pmatrix} \tau_j(\alpha_1) & \tau_k(\alpha_1) \\ \tau_j(\alpha_2) & \tau_k(\alpha_2) \end{pmatrix}, 
\]
for each cyclic permutation $(i,j,k)$ of $(1,2,3)$. Then each $c_i$ is a nonzero element of $L$ with
$$
h(c_i) \ll \max\{h(\alpha_1),h(\alpha_2)\} \ll h_*(\mu),
$$
and we have
\[
c_1\tau_1(\alpha_i) + c_2\tau_2(\alpha_i) + c_3\tau_3(\alpha_i) = 0, \quad \hbox{for $i = 1,2$.}
\]
Since the unit $v$ is a $\Z$-linear combination of $\alpha_1$ and $\alpha_2$, it also satisfies the linear relation   
\[
c_1\tau_1(v) + c_2\tau_2(v) + c_3\tau_3(v) = 0.
\]
After dividing all terms of this relation by $c_3\tau_3(v)$,
it follows from Theorem \ref{threeuniteq} that
\begin{equation}\label{htau}
	\max\{ h(\tau_1 (v) /\tau_3 (v)), h(\tau_2 (v)/\tau_3 (v) )  \} \ll h_*(\mu).
\end{equation}
Now the $\tau_1,\tau_2,\tau_3$ can be identified as restrictions to $L$ of three of the embeddings $\sigma_1,\ldots,\sigma_d$ 
of $K$ that we started with. Since we were allowed to choose any distinct embeddings of $\Q(\alpha_1/\alpha_2)$, 
we can assume $\sigma_d$ restricts to $\tau_3$ on $\Q(\alpha_1/\alpha_2)$.  
If $\sigma_1$ also restricts to $\tau_3$ on $\Q(\alpha_1/\alpha_2)$, this means
\[
\sigma_1(v/\alpha_2) = \sigma_d(v/\alpha_2), \quad \text{i.e.} \quad \frac{\sigma_1(\mathbf x)}{\sigma_d(\mathbf x)} 
= \frac{\sigma_1(\mu u_0)}{\sigma_d(\mu u_0)}\cdot  \frac{\sigma_d(\alpha_2)}{\sigma_1(\alpha_2)}.
\]
Recalling that $|\sigma_d(\mathbf x)| < 1$, this gives $\log |\sigma_1(\mathbf x)|  \ll h_* (\mu)$ and we conclude that 
$h(\bfx) \ll h_*(\mu)$.

Otherwise, we can assume $\sigma_1$ restricts to $\tau_1$, in which case from \eqref{htau} we get
\[
h \left( \frac{\sigma_1(v)}{\sigma_d(v)}  \right) \ll h_*(\mu), 
\]
and thus again $h(\bfx) \ll h_*(\mu)$.
Applying Theorem \ref{thxmu} completes the argument.

Next, we deal with the case where the subfield $\Q(\alpha_1/\alpha_2)$ of $L$ is quadratic. 
Since $L$ has a real embedding, $\Q(\alpha_1/\alpha_2)$ must in fact be real quadratic. 
We are looking for units in $L$ of the form $v = y_1\alpha_1+y_2\alpha_2$, where $y_1, y_2$ are in $\Z$. 
Thus $\tilde v = v/\alpha_2$ is quadratic. Recalling $u = u_0v$ we now write $u = u_0\alpha_2\tilde v$.


Suppose there exist indices $k,\ell$ with $1\le k<\ell \le d$ and $|\sigma_k(\tilde v)| < |\sigma_\ell(\tilde v)|$; in particular this means $\sigma_k$ and $\sigma_\ell$ restrict to the two distinct embeddings of the quadratic field $\Q(\alpha_1/\alpha_2)$. We have the pair of inequalities
\begin{align*}
|\sigma_k(\mu u_0 \alpha_2 \tilde v)| &\ge |\sigma_\ell(\mu u_0 \alpha_2 \tilde v)| \\
|\sigma_k(\tilde v)| &< |\sigma_\ell(\tilde v)|
\end{align*}
which yield
\begin{equation}\label{wrongorder}
1 < \left|   \frac{\sigma_\ell(\tilde v)}{\sigma_k(\tilde v)} \right|\le \left|  \frac{\sigma_k(\mu u_0 \alpha_2 )}{\sigma_\ell(\mu u_0 \alpha_2 )} \right|.
\end{equation}
Since $v$ is a unit in $L$ we have
\[
h(v) \ll \log \max_{1\le i,j \le d} \left| \frac{\sigma_i(\alpha_2\tilde v)}{\sigma_j(\alpha_2\tilde v)} \right| 
\ll  \log  \left|   \frac{\sigma_\ell(\tilde v)}{\sigma_k(\tilde v)} \right| + \log \max_{1\le i,j \le d} \left| \frac{\sigma_i(\alpha_2)}{\sigma_j(\alpha_2)} \right|
\]
and we get
\[
h(v) \ll \log \left|   \frac{\sigma_\ell(\tilde v)}{\sigma_k(\tilde v)} \right|  + h_*(\mu).
\]
Coupled with \eqref{wrongorder} this gives $h(v) \ll h_*(\mu)$ and thus $h(\bfx) \ll h_*(\mu)$ as desired. 
In the remaining case we must have
\beq \label{equal}
\sigma_1(\tilde v)  = \cdots = \sigma_{d/2}(\tilde v), \quad  \sigma_{d/2 + 1}(\tilde v) = \cdots = \sigma_d(\tilde v), 
\quad |\sigma_{d/2}(\tilde v)| \ge  |\sigma_{d/2 + 1}(\tilde v)|.     
\eeq
From Lemma \ref{boundsigma1} (with $n=4$) we have $|\sigma_1(\bfx)| \ll |\sigma_{d-4}(\bfx)|$ which gives   
\begin{equation}\label{only2}
|\sigma_1(\tilde v)| \ll \Bigl| \Bigl( \frac{\sigma_{d-4}(\mu u_0\alpha_2)}{\sigma_1(\mu u_0\alpha_2)} \Bigr)\sigma_{d-4}(\tilde v) \Bigr|.  
\end{equation}
In this case if the degree $d$ of the extension $K / \Q$ satisfies $d \ge 10$, then 
we necessarily have $\sigma_d(\tilde v) = \sigma_{d-4}(\tilde v)$ and 
we eventually get
\[
h(v) \ll \log\left| \frac{\sigma_1(\tilde v)}{\sigma_d(\tilde v)} \right| + h_*(\mu) \ll h_*(\mu).
\]
This again yields $h(\bfx) \ll h_*(\mu)$ and completes the last remaining case of Theorem \ref{thnf4}.  
Now, it only remains for us to deal with the case $d=8$ of Theorem \ref{thw4}. 
In view of \eqref{equal}, we have
$$
\left| \log\left| \frac{\sigma_j(\bfx)}{\sigma_k(\bfx)} \right|  \right| \ll h_*(\mu), \quad 5 \le j, k \le 8. 
$$
Consequently, by setting $P(X) = x_0 + x_1 X + \ldots + x_4 X^4$, 
there exists an effectively computable positive 
constant $C_7 = C_7(\xi)$ such that 
\begin{equation}\label{allm}
|m|^{-C_7} \le \frac{|P (\sigma_j (\xi))|}{|P (\sigma_k (\xi))|} \le |m|^{C_7}, \quad 5 \le j, k \le 8.
\end{equation}
Since the product of the $|P (\sigma_j (\xi))|$ over $j = 1, \ldots , 8$ is equal to $|m|$, we have
\begin{equation}\label{lastone}
|m|\cdot\prod_{j=1}^4 |P(\sigma_j(\xi))|^{-1} = \prod_{j=5}^8 |P(\sigma_j(\xi))|.
\end{equation}
Note that there exists an effectively computable positive constant $C_8 = C_8(\xi)$ such that
\[
|P (\sigma_1 (\xi))| \le C_8 H(P).
\]
Recalling that the $\sigma_j(\bfx) = P(\sigma_j(\xi))$ are ordered in decreasing order of their moduli, this means
\[
\prod_{j=1}^4 |P(\sigma_j(\xi))|^{-1} \ge |P(\sigma_1(\xi))|^{-4} \ge C_8^{-4} H(P)^{-4}.
\]
Meanwhile from \eqref{allm} we have
\[
\prod_{j=5}^8 |P(\sigma_j(\xi))| \le |m|^{3C_7}|P(\sigma_d(\xi))|^4.
\]
Combining the above two inequalities with \eqref{lastone} we get that
\[
C_8^{-4} |m| \cdot H(P)^{-4} \le |m|^{3C_7}|P(\sigma_d(\xi))|^4
\]
and therefore $|P(\xi)| \ge |P(\sigma_d(\xi))| \ge C_8^{-1}|m|^{-C_7} H(P)^{-1}$. We conclude by applying Lemma \ref{transf}.

\end{proof}

\section*{Acknowledgement}
We are grateful to the referee for a very meticulous reading and numerous suggestions  which helped us to 
considerably improve the presentation of the proofs and to correct several inaccuracies.  

\end{document}